\documentclass{article}

\usepackage{amsmath,amssymb,amsthm}
\usepackage{bm}
\usepackage{mathrsfs}
\usepackage[all]{xy}
\usepackage{booktabs}
\usepackage{fullpage}
\usepackage{hyperref}
\usepackage{mathtools}
\hypersetup{
colorlinks=false,
linkcolor=red,
citecolor=green
}

\newcommand{\rl}{\mathbb{R}}
\newcommand{\cx}{\mathbb{C}}

\newcommand{\ai}{\sqrt{-1}}

\newcommand{\cst}{\mathrm{const}}
\newcommand{\kah}{K\"ahler }
\newcommand{\ke}{K\"ahler--Einstein }

\newcommand{\ddbar}{\partial \bar{\partial}}

\theoremstyle{plain}% default
\newtheorem{theorem}{Theorem}[section]
\newtheorem{lemma}[theorem]{Lemma}
\newtheorem{proposition}[theorem]{Proposition}
\newtheorem{corollary}[theorem]{Corollary}
\theoremstyle{definition}
\newtheorem{definition}[theorem]{Definition}

\theoremstyle{definition}
\newtheorem{remark}[theorem]{Remark}

\newtheorem{problem}[theorem]{Problem}
\newtheorem{notation}[theorem]{Notation}
\newtheorem{question}[theorem]{Question}
\begin{document}

\title{Existence of twisted constant scalar curvature K\"ahler metrics with a large twist}
\author{Yoshinori Hashimoto}

\maketitle

\begin{abstract}
Suppose that there exist two K\"ahler metrics $\omega$ and $\alpha$ such that the metric contraction of $\alpha$ with respect to $\omega$ is constant, i.e.~$\Lambda_{\omega} \alpha = \text{const}$. We prove that for all large enough $R>0$ there exists a twisted constant scalar curvature K\"ahler metric $\omega'$ in the cohomology class $[ \omega ]$, satisfying $S(\omega' ) - R \Lambda_{\omega' }  \alpha = \text{const}$. We discuss its implication to $K$-stability and the continuity method recently proposed by X.X.~Chen.
\end{abstract}

\tableofcontents

\section{Introduction and the statement of the results}

\subsection{Twisted constant scalar curvature K\"ahler metrics and a continuity method}
Let $(X, \omega)$ be a compact \kah manifold without boundary. The existence of a ``canonical'' \kah metric, such as constant scalar curvature \kah (cscK) metrics satisfying $S(\omega) = \cst$, in a given cohomology class has been a central problem in \kah geometry. On the other hand, the cscK equation $S(\omega) = \cst$ is a fourth order fully nonlinear partial differential equation (PDE), and is difficult to solve in general.

A well-known method in solving a nonlinear PDE is the \textit{continuity method}. X.X.~Chen \cite{chentwisted} recently proposed a continuity method which can be seen as a generalisation to the cscK case of the Aubin--Yau continuity method that is well-known in \ke problems (cf.~\cite{aub, yaucy}). In this continuity method, we consider the following problem.

\begin{problem} {(X.X.~Chen \cite{chentwisted})} \label{tcontpath}
Let $t \in [0,1]$ be a parameter and $\alpha$ be a closed smooth positive real $(1,1)$-form. Find a family $\{ \omega_t \}_{t \in [0,1]}$ of smooth \kah metrics in a fixed cohomology class which satisfy
\begin{equation} \label{tcsckaycp}
t S(\omega_t) - (1-t) \Lambda_{\omega_t} \alpha = \cst ,
\end{equation}
where $\Lambda_{\omega_t}$ means the metric contraction with respect to $\omega_t$.
\end{problem}

Suppose that we write $I \subset [0,1]$ for the set of time parameter $t$ for which (\ref{tcsckaycp}) is solvable. The solution to Problem \ref{tcontpath} can be achieved in the following three steps.

\begin{enumerate}
\item Show that $I$ is \textit{nonempty}; i.e.~show that Problem \ref{tcontpath} can be solved at time $t=0$, or solve $\Lambda_{\omega_0} \alpha = \cst$ for $\omega_0$ when $\alpha$ is given. This is an interesting problem in its own right, in relation to the $J$-flow \cite{chenjflow, colsze, dervankeller, donjflow, lejsze, songweinkove, weinkovesurfaces, weinkovemab}. On the other hand we observe that, when we can assume that $\omega_0$ and $\alpha$ are in the same cohomology class, we can simply take $\omega_0 := \alpha$ to solve this equation.

\item Show that $I$ is \textit{open}; the {openness} of the problem at $t \in [0,1)$ means that, if (\ref{tcsckaycp}) is solvable at time $t \in [0,1)$, there exists $\epsilon_0 >0$ such that (\ref{tcsckaycp}) can be solved for each $t' \in (t - \epsilon , t + \epsilon) \cap [0,1)$ for all $0< \epsilon \le \epsilon_0$.

\item Show that $I$ is \textit{closed}; the {closedness} of the problem means that, for every convergent sequence $\{ t_i \}_i$ in $[0,1]$, if (\ref{tcsckaycp}) can be solved for each $t_i$ then it can be solved for $t_{\infty} = \lim_i t_i$.
\end{enumerate}

Observe that the solution of Problem \ref{tcontpath} at $t=1$ is precisely equal to solving the cscK equation $S(\omega_1) = \cst$; we thus reduced the problem of finding a cscK metric to establishing the above, hopefully more manageable, three steps in Problem \ref{tcontpath}.

%solve the ``easier'' problem of solving $\Lambda_{\omega_0} \alpha = \cst$ at $t=0$, and reduce the difficult problem of finding a cscK metric to solving the openness and closedness of Problem \ref{tcontpath}.

The main result of this paper is the openness at $t =0$ (cf.~Remark \ref{opnt0tn0}), stated as follows, which solves the question \cite[Question 1.6]{chentwisted} posed by X.X.~Chen in the affirmative.

\begin{theorem} \label{mrtcsck}
Suppose that we have two \kah metrics $\omega$ and $\alpha$ satisfying $\Lambda_{\omega} \alpha = \cst$. Then there exists a constant $R(\omega , \alpha) >0$ depending only on $\omega$ and $\alpha$, so that for all $R \ge R(\omega , \alpha)$ there exists $\phi  \in C^{\infty} (X , \rl)$ such that $\omega_{\phi } := \omega + \ai \ddbar \phi$ satisfies $S(\omega_{\phi}) -  \Lambda_{\omega_{\phi}} (R \alpha) = \cst$.% for all $R \ge R(\omega , \alpha)$.
\end{theorem}

\begin{remark}
As the preparation of this paper was nearing completion, the author learned that Yu Zeng \cite{zeng} independently proved a special case of the above theorem when we take $\alpha = \omega$.
\end{remark}

%When $t \neq 0,1$, we thus need to study the following equation

%It is well-known that such a canonical metric does not always exist, and a celebrated programme initiated by Yau, Tian, and Donaldson conjectures that the existence of cscK metric should be equivalent to some form of algebro-geometric ``stability''. It is widely believed that $K$-stability, or technical modifications thereof, should be the ``right'' stability condition.

%Given another \kah metric $\alpha$, we study the equation

Recall that, for a closed positive real $(1,1)$-form $\alpha$, a \kah metric $\omega_{\phi}$ satisfying the equation
\begin{equation} \label{tcsck}
S(\omega_{\phi}) - \Lambda_{\omega_{\phi}} \alpha = \cst 
\end{equation}
is said to be of \textbf{$\alpha$-twisted constant scalar curvature K\"ahler} or \textbf{$\alpha$-twisted cscK}. Twisted cscK metrics were applied to the study of (genuine) cscK metrics by means of adiabatic construction \cite{finesurfaces, finefibration}, and are an interesting object in their own right \cite{dertwisted, kellertwisted, stoppatwisted}. It is also known that an $\alpha$-twisted cscK metric is unique in each cohomology class \cite[Theorem 4.5]{bb}, and this was used to prove the uniqueness of cscK metrics modulo automorphisms \cite[Theorem 1.3]{bb}. Naively re-phrasing, Theorem \ref{mrtcsck} ensures the existence of $R \alpha$-twisted cscK metrics on any compact \kah manifold, assuming $\Lambda_{\omega} \alpha = \cst$ and taking $R>0$ to be sufficiently large; in particular, this implies that we can always find an $R \alpha$-twisted cscK metric on any compact \kah manifold in the cohomology class $[ \alpha ]$, by taking $\omega = \alpha$.

%which is the main object of study of this paper. $\omega_{\phi}$ satisfying the above equation is called an \textbf{$\alpha$-twisted constant scalar curvature K\"ahler} metric or \textbf{$\alpha$-twisted cscK} metric, and there has been a lot of work on this. 

We shall also see that the proof of Theorem \ref{mrtcsck} can be easily applied to prove the openness at $t \in (0,1)$; recall on the other hand that the openness at $t \in (0,1)$ was also obtained independently by X.X.~Chen \cite[Theorem 1.5]{chentwisted}. In fact, we can prove the following slightly stronger result.

\begin{corollary} \label{cormroptwa}
Suppose that $\omega$ is $\alpha'$-twisted cscK. Then, if $\alpha' - \alpha$ is sufficiently small in the $C^{\infty}$-norm, there exists an $\alpha$-twisted cscK metric in the cohomology class $[\omega]$.
\end{corollary}

Thus, as in the \ke case, we see that proving the closedness of $I$ is the hardest part in solving Problem \ref{tcontpath}.

%The problem at $t=0$ can be easily solved; we can just take $\omega_0 := \alpha$. If we can show the \textit{openness} and \textit{closedness}, we can prove the existence of a cscK metric. The openness of the problem at $t_0 \in [0,1)$ means that for $t_0 \in [0,1)$ there exists $\epsilon_0 >0$ such that (\ref{tcsckaycp}) can be solved for $t \in (t_0 - \epsilon , t_0 + \epsilon) \cap [0,1)$ for all $0< \epsilon \le \epsilon_0$. The closedness of the problem means that, for any convergent sequence $\{ t_i \}_i$ in $[0,1]$, if (\ref{tcsckaycp}) can be solved for each $t_i$ then it can be solved for $t_{\infty} = \lim_i t_i$.

\begin{remark}
One way of obtaining a twisted cscK metric is to solve the following ``un-traced'' version of the equation (\ref{tcsck}):
\begin{equation} \label{aytcsckut}
\mathrm{Ric} (\omega_{\phi}) - R \alpha = \cst . \omega_{\phi}.
\end{equation}
This is solvable for any positive twist $\alpha$ by the Aubin--Yau theorem \cite{aub, yaucy}, when $R>0$ is chosen to be sufficiently large\footnote{See also \cite{bermantherm, songtian1, songtian2} for related results in which we do not necessarily take $R$ to be large.}. This was used by Fine \cite{finefibration} to construct cscK metrics on fibred \kah manifolds by means of adiabatic construction. However, the equation (\ref{aytcsckut}) means that the cohomology class of $\omega_{\phi}$ must be a constant multiple of $ R [ \alpha ] + c_1 (K_X)$, where $K_X$ is the canonical bundle of $X$. Theorem \ref{mrtcsck} removes this restriction on the cohomology class, assuming instead that the equation $\Lambda_{\omega} \alpha = \cst$ should be satisfied. An advantage of this can be seen when we choose $\omega = \alpha$; we can find a twisted cscK metric \textit{in the cohomology class} $[\omega] = [\alpha]$. % \cite[Theorem 4]{yaucy}
\end{remark}

\begin{remark} \label{opnt0tn0}
The openness of Problem \ref{tcontpath} at $t=0$ is different from the one at $0<t<1$ in a rather significant way due to the following fact: at $t=0$, the equation $\Lambda_{\omega_0} \alpha = \cst$ is a second order fully nonlinear PDE in the \kah potential, whereas at $0<t<1$, the equation $t S(\omega_t) - (1-t) \Lambda_{\omega_t} \alpha = \cst$ is a \textit{fourth} order fully nonlinear PDE. Thus, to prove the openness at $t=0$, we must deal with this ``jump'' in the order of the equation in question, as opposed to the case $0<t<1$.%which the author finds nontrivial.%cannot be done by e.g.~applying Banach space inverse function theorem.%; the main point of the proof of Theorem \ref{mrtcsck} is to start from a simple observation as in Proposition \ref{trivsolprop}, improve it by an iterative construction, so that we can reduce to the Banach space inverse function theorem.
\end{remark}

\begin{remark} \label{remavsclal}
Observe that the average $\bar{S}$ of the scalar curvature $S(\omega_{\phi})$ with respect to the volume form $\omega_{\phi}^n / n!$ is equal to the cohomological number $- n \int_X c_1 (K_X) \wedge [\omega^{n-1}_{\phi}] / \int_X [\omega^{n}_{\phi}]$ by the Chern--Weil theory. Since $\alpha$ is closed, the average $c$ of $\Lambda_{\omega_{\phi}} \alpha$ with respect to the volume form $\omega_{\phi}^n / n!$ is also determined by the cohomology classes of $\alpha$ and $\omega_{\phi}$ as $c= n \int_X [\alpha] \wedge [\omega^{n-1}_{\phi}] / \int_X [\omega^{n}_{\phi}]$.
\end{remark}

\subsection{Relationship to $K$-stability}

A central problem in \kah geometry in recent years has been the connection between the existence of cscK metrics and a notion of algebro-geometric stability called $K$-stability; a conjecture called \textbf{Donaldson--Tian--Yau conjecture} \cite{dontoric, tian97, yauprob} states that, when the automorphism group of $X$ is discrete\footnote{When the automorphism group of $X$ is not discrete, we consider a notion called $K$-\textit{polystability}, but we do not discuss this in detail here.}, there exists a cscK metric in the first Chern class $c_1 (L)$ of an ample line bundle $L$ if and only if the polarised \kah manifold $(X,L)$ is $K$-stable\footnote{It is expected that the notion of $K$-stability may have to be refined for this conjecture to hold, cf.~\cite{szefilt}.}.

We will be very brief in recalling the notion of $K$-stability here, and the reader is referred to \cite{dontoric, tian97} for more details. A \textbf{test configuration} for a polarised \kah manifold $(X,L)$, written $(\mathcal{X} , \mathcal{L})$, is a flat family $\pi : \mathcal{X} \to \cx$ over $\cx$ with an equivariant $\cx^*$-action lifting to the total space of a line bundle $\mathcal{L}$ such that $\pi^{-1} (1)$ is isomorphic to $(X,L^{\otimes r})$, and $r$ is called the exponent of the test configuration $(\mathcal{X} , \mathcal{L})$. We can define a rational number called the \textbf{Donaldson--Futaki invariant} $DF(\mathcal{X} , \mathcal{L})$ for each $(\mathcal{X} , \mathcal{L})$ as in \cite[\S 2.1]{dontoric}, and $(X,L)$ is said to be \textbf{$K$-stable} if $DF(\mathcal{X} , \mathcal{L}) \ge 0$ for every test configuration and $DF(\mathcal{X} , \mathcal{L}) >0$ for every ``nontrivial'' test configuration\footnote{There are subtleties associated to this formalism, e.g.~as to what ``trivial'' test configurations should mean, and the reader is referred to \cite{bhj, dertwisted, lixu, szefilt} for more details.}.

%We briefly recall that a polarised \kah manifold $(X,L)$, a pair of a \kah manifold and an ample line bundle $L$, is said to be \textbf{$K$-stable} \cite{dontoric, tian97} if for any nontrivial test configuration $(\mathcal{X} , \mathcal{L})$ of $(X,L)$, a quantity called the \textbf{Donaldson--Futaki invariant} $DF(\mathcal{X} , \mathcal{L})$ is strictly positive, where we recall that a \textbf{test configuration} $(\mathcal{X} , \mathcal{L})$ is a flat family $\pi : \mathcal{X} \to \cx$ of varieties over $\cx$ with an equivariant $\cx^*$-action lifting to the total space of the line bundle $\mathcal{L}$ such that $\pi^{-1} (1)$ is isomorphic to $(X,L)$; there are many subtleties associated to this formalism, e.g.~as to what ``trivial'' test configurations should mean, and the reader is referred to \cite{bhj, dertwisted, lixu, szefilt} for more details.

Dervan \cite{dertwisted} introduced the \textbf{minimum norm} $|| \mathcal{X}||_m$ of a test configuration $(\mathcal{X} , \mathcal{L})$ and proved that the existence of an $\alpha$-twisted cscK metric in the cohomology class $c_1 (L)$ implies the \textit{uniform twisted $K$-stability} of $(X,L, \frac{1}{2} [\alpha])$ \cite[Theorem 1.1]{dertwisted}. Note also that the minimum norm agrees, up to a constant multiple, with the non-Archimedean $J$-functional defined by Boucksom--Hisamoto--Jonsson \cite{bhj}. It turns out that when the cohomology class of $\alpha$ is a (positive) constant multiple of $c_1 (L)$, uniform twisted $K$-stability of $(X,L , \frac{1}{2} [\alpha])$ implies that there exists a constant $\tilde{R}>0$ such that $DF(\mathcal{X} , \mathcal{L}) + \tilde{R} || \mathcal{X} ||_m >0$ \cite[Remark 2.34]{dertwisted}. Recalling that it is always possible to solve $\Lambda_{\omega} \alpha = \cst$ for $\omega$ in the cohomology class $[ \alpha ] $, by taking $\omega = \alpha$, Theorem \ref{mrtcsck} implies the following result.

%The above theorem implies the following uniform lower bound of the Donaldson--Futaki invariant $DF(\mathcal{X} , \mathcal{L})$, by taking $\omega_{\phi} , \alpha \in c_1(L)$ of an ample line bundle $L$.

\begin{corollary} \label{cormrkst}
There exists a constant $\tilde{R}>0$, which depends only on $(X,L)$, such that $DF(\mathcal{X} , \mathcal{L}) > - \tilde{R} || \mathcal{X} ||_{m}$ for any test configuration $(\mathcal{X} , \mathcal{L})$ for $(X,L)$.% where $|| \mathcal{X} ||_{m}$ is the minimum norm of $\mathcal{X}$. %(or on an arbitrarily chosen \kah metric $\omega \in c_1 (L)$)
\end{corollary}

\begin{remark}
In fact this lower bound itself can be obtained from Proposition \ref{trivsolprop}, which is the starting point of our proof of Theorem \ref{mrtcsck}.%a straightforward observation as stated in Proposition \ref{trivsolprop}.
\end{remark}

\begin{remark}
We now recall that there is another norm $|| \mathcal{X} ||_2$ for a test configuration $(\mathcal{X} , \mathcal{L})$, called the \textbf{$L^2$-norm} of $(\mathcal{X} , \mathcal{L})$. It is known \cite{bhj, dertwisted} that $|| \mathcal{X} ||_2 =0$ if and only if $|| \mathcal{X} ||_m =0$, but they are \textit{not} Lipschitz equivalent \cite{bhj}. Recall that the $L^2$-norm appears in the lower bound of the Calabi functional $Ca (\omega) := \int_X (S (\omega) - \bar{S})^2 \frac{\omega^n}{n!}$, which takes the form
\begin{equation} \label{lbcalfcn}
 \inf_{\omega \in c_1(L)} \int_X (S (\omega) - \bar{S})^2 \frac{\omega^n}{4 \pi n!} \ge \sup_{(\mathcal{X}, \mathcal{L})} \left(  -   \frac{DF(\mathcal{X} , \mathcal{L})}{||\mathcal{X}||_2} \right)
\end{equation}
as established by Donaldson \cite{donlb}. It immediately follows that we have $DF(\mathcal{X} , \mathcal{L}) \ge - \tilde{R} || \mathcal{X} ||_{2}$ by taking $\tilde{R}$ to be the left hand side of the inequality (\ref{lbcalfcn}). Thus, Corollary \ref{cormrkst} can be seen as a \textit{minimum-norm version} of this particular consequence of the lower bound of the Calabi functional.
\end{remark}

\subsection{Some open problems}

In view of Problem \ref{tcontpath}, it is natural to define the following quantity
\begin{equation*}
\tilde{R}_{\alpha} := \inf \{ R \ge 0 \mid \exists \phi  \in C^{\infty} (X , \rl) \text{ s.t. } S(\omega_{\phi}) -  \Lambda_{\omega_{\phi}} (R \alpha) = \cst , \ \omega \in c_1 (L) \} ,
\end{equation*}
as introduced by X.X.~Chen \cite[Definition 1.6]{chentwisted}, which is analogous to the $R(X)$ invariant defined by Sz\'ekelyhidi \cite{szerx} for Fano manifolds. Given Corollary \ref{cormrkst}, it seems natural to consider the following question.

\begin{question} \label{tcpekutc}
Does there exist a test configuration $(\tilde{\mathcal{X}} , \tilde{\mathcal{L}})$ with $|| \tilde{\mathcal{X}} ||_m >0$ such that $DF(\tilde{\mathcal{X}} , \tilde{\mathcal{L}}) = - \tilde{R}_{\alpha} || \tilde{\mathcal{X}} ||_{m}$?
\end{question}

Since the Donaldson--Futaki invariant and the minimum norm can be defined in terms of algebro-geometric data, we can ask if $\tilde{R}_{\alpha}$ in the above can be written without referring to the particular choice of twist $\alpha$, and potentially in terms of algebro-geometric language. We thus ask the following question, as conjectured by X.X.~Chen.

\begin{question} {(X.X.~Chen \cite[Conjecture 1.17]{chentwisted})} \label{qindtwiscc}
For any two closed positive real $(1,1)$-forms $\alpha$ and $\beta$ in the same cohomology class, do we have $\tilde{R}_{\alpha} = \tilde{R}_{\beta}$?
\end{question}

If we can solve Question \ref{tcpekutc} in the affirmative, we see that $\tilde{R}_{\alpha} >0$ would imply that $(X,L)$ cannot be $K$-stable, since $(\tilde{\mathcal{X}} , \tilde{\mathcal{L}})$ would provide a test configuration whose Donaldson--Futaki invariant is strictly negative. Noting that there cannot be a cscK metric in $c_1(L)$ if $\tilde{R}_{\alpha} >0$, this would provide an evidence in support of the Donaldson--Tian--Yau conjecture.

%On the other hand, the solution of Question \ref{tcpekutc} could be out of reach at the moment. However, the author believes that the following problem is more approachable.

%The affirmative solution of Question \ref{qindtwiscc} would allow us to express the quantities in the equation $DF(\tilde{\mathcal{X}} , \tilde{\mathcal{L}}) = - \tilde{R}_{\alpha} || \tilde{\mathcal{X}} ||_{m}$ without referring to the particular choice of twist $\alpha$, and potentially in terms of purely algebro-geometric language. The affirmative solution of Question \ref{tcpekutc} is closely related to the resolution of the Donaldson--Tian--Yau conjecture; $\tilde{R}_{\alpha} >0$ implies that $(X,L)$ cannot admit a cscK metric, and the above $(\tilde{\mathcal{X}} , \tilde{\mathcal{L}})$ provides a $K$-unstable test configuration. Although solving the above question may be out of reach at the moment, I believe that the following problem is more approachable.
%\begin{problem}
%When $X = \mathrm{Bl}_{\mathrm{pt}} \prj^2$ (the projective plane blown up at a point), which does not admit a cscK metric in any \kah class, compute explicitly $\tilde{R}_{\alpha}$ for a given twist $\alpha$, for each polarisation.
%\end{problem}

%A progress on this problem will appear in the subsequent work of the author.

\subsection*{Acknowledgements}

The author is very grateful to Xiuxiong Chen for kindly acknowledging this work in \cite{chentwisted} although at that time this work was only in a draft state, and also for encouragements. The author thanks Ruadha\'i Dervan for many helpful discussions on twisted cscK metrics and helpful comments. Last but not least, thanks are due to the author's supervisor Jason Lotay for helpful comments.

%, and particularly for pointing out a connection to the modified Futaki invariant in the semipositive case.

\section{Linearisation of the twisted cscK equation} \label{linotcscke}

%We use the openness above to show that we can find $\phi$ such that $\omega_{\phi}$ is $\alpha$-twisted cscK.

%It suffices to show that the first eigenvalue of $- \mathfrak{D}_{\omega}^* \mathfrak{D}_{\omega} \phi + R F_{\alpha} (\phi)$ is bounded away from zero, which implies that the operator norm of the inverse is bounded. This required bound can be achieved by taking $R$ to be very large and recalling the Poincar\'e inequality ($- \mathfrak{D}_{\omega}^* \mathfrak{D}_{\omega} \phi + R F_{\alpha} (\phi)$ is self-adjoint and elliptic, so we can take the complete orthonormal basis in $L^2$ consisting of eigenfunctions). We can then use the quantitative version of Banach space inverse function theorem, as in Theorem 5.3 of Fine's thesis.

Suppose that we write $\omega_t := \omega + t \ai \ddbar \phi$. Recall that the scalar curvature $S(\omega_t)$ of $\omega_t$ is defined by $S(\omega_t) := \Lambda_{\omega_t} \mathrm{Ric} (\omega_t)$, where $\mathrm{Ric} (\omega_t) : = - \ai \ddbar \log \det (\omega_t)$ is the Ricci form of $\omega_t$. Locally, writing $\omega_t = \sum_{i,j} g_{i \bar{j}, t}  \ai d z_i \wedge d \bar{z}_j$ in local holomorphic coordinates $(z_1 , \dots , z_n)$, we have
\begin{equation*}
S(\omega_t) = - \sum_{i,j} g^{i \bar{j} }_t  \frac{\partial^2}{\partial z_i \partial \bar{z}_j} \log \det (g_{k \bar{l} , t})
\end{equation*}
where $g^{i \bar{j} }_t$ is the inverse matrix of $g_{i \bar{j}, t} $. We find by direct computation that
\begin{equation*}
\left. \frac{d}{dt} \right|_{t=0} S(\omega_t) = -\Delta^2_{\omega} \phi - (\mathrm{Ric} (\omega) , \ai \ddbar \phi)_{\omega} %$\Delta_{\omega}:= - \bar{\partial} \bar{\partial}^* -  \bar{\partial}^* \bar{\partial}$.
\end{equation*}
where $\Delta_{\omega}$ is the $\bar{\partial}$-Laplacian and $(,)_{\omega}$ is a (pointwise) inner product on the space of 2-forms defined by $\omega$. It is well-known (cf.~\cite{lebsim}) that it can also be written as
\begin{equation*}
\left. \frac{d}{dt} \right|_{t=0} S(\omega_t) =- \mathfrak{D}_{\omega}^* \mathfrak{D}_{\omega} \phi + (\partial S(\omega) , \bar{\partial} \phi)_{\omega}
\end{equation*}
where $\mathfrak{D}_{\omega}: C^{\infty} (X , \rl) \to C^{\infty} (T^{1,0} X \otimes \Omega^{0,1}(X))$ is an operator defined by $\mathfrak{D}_{\omega} \phi := \bar{\partial} (\mathrm{grad}^{1,0}_{\omega} \phi)$, and $\mathfrak{D}^*_{\omega}$ is the formal adjoint with respect to $\omega$. Observe that the kernel of $\mathfrak{D}_{\omega}^* \mathfrak{D}_{\omega}$, which is equal to the kernel of $ \mathfrak{D}_{\omega}$ as $X$ is compact without boundary, is equal to the set of functions whose $(1,0)$-part of the gradient is a holomorphic vector field.

Now, let $\alpha$ be a closed positive $(1,1)$-form. Straightforward computation yields
\begin{equation*}
\left. \frac{d}{dt} \right|_{t=0} \Lambda_{\omega_t} \alpha = - (\alpha , \ai \ddbar \phi)_{\omega} .
\end{equation*}
We thus get
\begin{align*}
&\left. \frac{d}{dt} \right|_{t=0} (S(\omega_t) - \Lambda_{\omega_t} \alpha) \\
&= - \mathfrak{D}_{\omega}^* \mathfrak{D}_{\omega} \phi + (\partial S(\omega) , \bar{\partial} \phi)_{\omega} - (\partial (\Lambda_{\omega} \alpha) , \bar{\partial} \phi)_{\omega} + (\alpha , \ai \ddbar \phi)_{\omega}  + (\partial (\Lambda_{\omega} \alpha) , \bar{\partial} \phi)_{\omega} .
\end{align*}
%\begin{align*}
%\left. \frac{d}{dt} \right|_{t=0} (S(\omega_t) - \Lambda_{\omega_t} \alpha) &= - \Delta^2_{\omega} \phi - \mathrm{Ric}(\omega)^{k \bar{l}} \frac{\partial^2 \phi}{\partial z_k \partial \bar{z}_l} + g^{i \bar{l}} g^{m \bar{j}} \alpha_{i \bar{j}} \frac{\partial^2 \phi}{\partial z_m \partial \bar{z}_l} \\
%&= - \mathfrak{D}_{\omega}^* \mathfrak{D}_{\omega} \phi + (\partial S(\omega) , \bar{\partial} \phi)_{\omega} + (\alpha , \ai \ddbar \phi)_{\omega} - (\partial (\Lambda_{\omega} \alpha) , \bar{\partial} \phi)_{\omega} + (\partial (\Lambda_{\omega} \alpha) , \bar{\partial} \phi)_{\omega}
%\end{align*}
Note that, if $\omega$ is an $\alpha$-twisted cscK metric, i.e.~satisfies $S(\omega) - \Lambda_{\omega} \alpha = \cst$, we thus have
\begin{equation} \label{linatcsckeq}
\left. \frac{d}{dt} \right|_{t=0} (S(\omega_t) - \Lambda_{\omega_t} \alpha) = - \mathfrak{D}_{\omega}^* \mathfrak{D}_{\omega} \phi +  (\alpha , \ai \ddbar \phi)_{\omega} + (\partial (\Lambda_{\omega} \alpha) , \bar{\partial} \phi)_{\omega} ,
\end{equation}
and hence it seems natural to make the following definition.
\begin{definition}
Given two \kah metrics $\omega$ and $\alpha$, we define an operator $F_{\omega, \alpha}: C^{\infty} (X , \rl) \to C^{\infty} (X , \rl)$ by
\begin{equation*}
F_{\omega , \alpha} (\phi) :=  (\alpha , \ai \ddbar \phi)_{\omega} + (\partial (\Lambda_{\omega} \alpha) , \bar{\partial} \phi)_{\omega} .
\end{equation*}

\end{definition}

\begin{lemma} \label{lemlinojfcsasorlok}
$F_{\omega, \alpha}$ is a complex self-adjoint second order elliptic linear operator which satisfies
\begin{equation*}
\int_X \psi F_{\omega, \alpha} (\phi) \frac{\omega^n}{n!} = - \int (\overline{\xi_{\psi}} , \xi_{\phi} )_{\alpha}  \frac{\omega^n}{n!} ,
\end{equation*}
where $\xi_{\phi}:= ({\partial } \phi)^{\sharp, \omega}$ (resp.~$\xi_{\psi}:= ({\partial} \psi)^{\sharp, \omega}$) is the $\omega$-metric dual of ${\partial} \phi$ (resp.~${\partial} \psi$). In particular, $\ker F_{\omega, \alpha}$ is the set of constant functions.
\end{lemma}

\begin{proof}
 It is immediate that $F_{\omega , \alpha}$ is a second order elliptic linear operator, since $\alpha$ is strictly positive. By recalling some well-known identities (see e.g.~\cite[Lemma 4.7]{szebook}), we compute
\begin{align*}
&\int_X \psi F_{\omega , \alpha} (\phi) \frac{\omega^n}{n!}  \\
%&= \int_X \left( \psi (\alpha , \ai \ddbar \phi)_{\omega} + \psi (\partial (\Lambda_{\omega} \alpha) , \bar{\partial} \phi)_{\omega} \right)\frac{\omega^n}{n!} \\
&= \int_X \psi (\Lambda_{\omega} \alpha \Delta_{\omega} \phi) \frac{\omega^n}{n!} - \int_X \psi \alpha \wedge \ai \ddbar \phi \wedge \frac{\omega^{n-2}}{(n-2)!} + \int_X \psi \ai \partial (\Lambda_{\omega} \alpha) \wedge \bar{\partial } \phi \wedge \frac{\omega^{n-1}}{(n-1)!} .
%&=\int_X \phi \Delta_{\omega} (\psi (\Lambda_{\omega} \alpha))   \frac{\omega^n}{n!} + \int_X \ai \partial \psi \wedge \alpha \wedge \bar{\partial} \phi \wedge \frac{\omega^{n-2}}{(n-2)!} \\
%&\ \ \ \ \ \ +  \int_X \ai \bar{\partial} \psi \wedge \partial (\Lambda_{\omega} \alpha) \wedge \phi \frac{\omega^{n-1}}{(n-1)!} - \int_X \psi \ai \ddbar (\Lambda_{\omega} \alpha) \wedge  \phi \frac{\omega^{n-1}}{(n-1)!} 
\end{align*}

Note that, integrating by parts, we have
\begin{equation*}
\int_X \psi (\Lambda_{\omega} \alpha \Delta_{\omega} \phi) \frac{\omega^n}{n!} = - \int_X ( \Lambda_{\omega} \alpha) \ai \partial \psi \wedge \bar{\partial} \phi \wedge \frac{\omega^{n-1}}{(n-1)!} - \int_X  \psi \ai \partial (\Lambda_{\omega} \alpha) \wedge \bar{\partial} \phi \wedge \frac{\omega^{n-1}}{(n-1)!}  ,
\end{equation*}
and
\begin{equation*}
- \int_X \psi \alpha \wedge \ai \ddbar \phi \wedge \frac{\omega^{n-2}}{(n-2)!} = \int_X \ai \partial \psi \wedge \alpha \wedge \bar{\partial} \phi \wedge \frac{\omega^{n-2}}{(n-2)!} ,\end{equation*}
since $\alpha$ is closed. Thus
\begin{align*}
&\int_X \psi F_{\omega , \alpha} (\phi) \frac{\omega^n}{n!}  \\
&=  - \int_X ( \Lambda_{\omega} \alpha) \ai \partial \psi \wedge \bar{\partial} \phi \wedge \frac{\omega^{n-1}}{(n-1)!} +  \int_X \ai \partial \psi \wedge \alpha \wedge \bar{\partial} \phi \wedge \frac{\omega^{n-2}}{(n-2)!}  \\
&=  - \int_X ( \Lambda_{\omega} \alpha) \Lambda_{\omega} (\ai \partial \psi \wedge \bar{\partial} \phi ) \frac{\omega^{n}}{n!} + \int_X \ai \partial \psi  \wedge \bar{\partial} \phi \wedge \alpha \wedge \frac{\omega^{n-2}}{(n-2)!}  \\
&= -\int (\alpha, \ai \partial \psi \wedge \bar{\partial} \phi )_{\omega} \frac{\omega^n}{n!} \\
&= - \int (\overline{\xi_{\psi}} , \xi_{\phi} )_{\alpha}  \frac{\omega^n}{n!} 
\end{align*}
where we wrote $\xi_{\phi}:= ({\partial } \phi)^{\sharp, \omega}$ (resp.~$\xi_{\psi}:= ({\partial} \psi)^{\sharp, \omega}$) for the $\omega$-metric dual of ${\partial} \phi$ (resp.~${\partial} \psi$).

We thus get
\begin{equation*}
\int_X \psi F_{\omega , \alpha} (\phi) \frac{\omega^n}{n!} = \int_X \overline{\phi F_{\omega , \alpha} (\psi)} \frac{\omega^n}{n!} 
\end{equation*}
and hence $F_{\omega , \alpha}$ is (complex) self-adjoint. We also see that 
\begin{equation*}
\int_X \phi F_{\omega , \alpha} (\phi) \frac{\omega^n}{n!} =  - \int_X || \xi_{\phi} ||^2_{\alpha}  \frac{\omega^n}{n!} <0
\end{equation*}
for every non-constant function $\phi$, since $\alpha$ is positive definite. Thus $F_{\omega , \alpha} (\phi) =0$ if and only if $\phi$ is constant.

\end{proof}

\section{Proof of Theorem \ref{mrtcsck}}

We follow the line of argument similar to the one in the paper of Fine \cite{finesurfaces} or LeBrun--Simanca \cite{lebsim}; we construct approximate solutions to the $R \alpha$-twisted cscK equation (\S \ref{trivsolprop}), and then apply the Banach space inverse function theorem to get the genuine solution (\S \ref{appinvfnth}).

\subsection{Construction of approximate solutions}

We start with the following observation.
\begin{proposition} \label{trivsolprop}
Let $\omega$ and $\alpha$ be \kah metrics satisfying $\Lambda_{\omega} \alpha = \cst$, and let $G(\omega)$ be the solution to $\Delta_{\omega} G(\omega) = S(\omega) - \bar{S}$. Then $\omega$ is $R \alpha'$-twisted cscK if we define $\alpha' := \alpha + \ai \ddbar G(\omega) / R$, which is strictly positive if $R >0$ is chosen to be sufficiently large. 
\end{proposition}

Thus, almost by tautology, we get an $R \alpha'$-twisted cscK metric for $\alpha' \in [\alpha]$ that is different from $\alpha$ by order $1/R$. Our aim in what follows is to ``improve'' this observation ``order by order'', so that we get an $R \alpha_m$-twisted cscK metric for $\alpha_m \in [ \alpha ]$ that is different from $\alpha$ by order $1/R^m$, say.

Suppose $\Lambda_{\omega} \alpha = \cst$. Then we have the trivial
\begin{equation*}
S(\omega) - R \Lambda_{\omega} \alpha = \cst + (S(\omega) - \bar{S}) .
\end{equation*}

Now consider $\omega_1 : = \omega + \ai \ddbar \phi_1 / R$. Then, expanding in $1/R$, we get
\begin{equation*}
S(\omega_1) - R \Lambda_{\omega_1} \alpha = \cst + (S(\omega) - \bar{S})  + (\alpha , \ai \ddbar \phi_1)_{\omega} + O(1/R).
\end{equation*}
%\begin{align*}
%S(\omega_1) - R \Lambda_{\omega_1} \alpha &= S(\omega) - R \Lambda_{\omega} \alpha + \frac{1}{R} \mathbb{L}_{\omega} \phi_1 + (\alpha , \ai \ddbar \phi_1)_{\omega} + O(1/R) \\
%&= \cst + (S(\omega) - \bar{S})  + (\alpha , \ai \ddbar \phi_1)_{\omega} + O(1/R)
%\end{align*}
We now wish to choose $\phi_1 \in C^{\infty} (X , \rl)$ so that $(S(\omega) - \bar{S})  + (\alpha , \ai \ddbar \phi_1)_{\omega}$ becomes constant. This can be achieved by the following lemma.

\begin{lemma} \label{lemsolvoboR}
Suppose $\Lambda_{\omega} \alpha = \cst$ and that the average of $f \in C^{\infty} (X , \rl)$ over $X$ with respect to $\omega$ is $0$. Then there exists $\phi \in C^{\infty} (X , \rl)$ such that $(\alpha , \ai \ddbar \phi)_{\omega} = f$.
\end{lemma}

\begin{proof}
First of all, $\Lambda_{\omega} \alpha = \cst$ implies $(\alpha , \ai \ddbar \phi)_{\omega} = F_{\omega , \alpha} (\phi)$. Note also that (cf.~\cite[Lemma 4.7]{szebook})
\begin{align*}
\int_X F_{\omega ,\alpha} (\phi) \frac{\omega^n}{n!} = \int_X (\alpha , \ai \ddbar \phi)_{\omega} \frac{\omega^n}{n!} &= \int_X \Lambda_{\omega} \alpha \Delta_{\omega} \phi \frac{\omega^n}{n!} - \int_X \alpha \wedge \ai \ddbar \phi \wedge \frac{\omega^{n-2}}{(n-2)!} \\
&=0
\end{align*}
since $\Lambda_{\omega} \alpha = \cst$ and $\alpha$ is closed. This means that in order for the equation $(\alpha , \ai \ddbar \phi)_{\omega} = f$ to hold, it is necessary that the average of $f$ is zero.

Suppose that we write $C^{\infty} (X , \rl)_0$ for the set of smooth functions whose average (with respect to $\omega$) is zero. We get the claimed result if the operator $F_{\omega , \alpha} :  C^{\infty} (X , \rl)_0 \to C^{\infty} (X , \rl)_0$ is surjective, assuming $\Lambda_{\omega} \alpha = \cst$. We pass to the Sobolev completion $L^2_p$ of $C^{\infty} (X , \rl)_0$. Since the operator $F_{\omega , \alpha} : L^2_p \to L^2_{p-2}$ is elliptic and $X$ is compact without boundary, it is Fredholm. Lemma \ref{lemlinojfcsasorlok} shows that $F_{\omega , \alpha}$ is self-adjoint and that the kernel of $F_{\omega , \alpha}$ is trivial, and hence by the Fredholm alternative we conclude that $F_{\omega , \alpha} : L^2_p \to L^2_{p-2}$ is surjective.

% Lemma \ref{lemlinojfcsasorlok} shows that the operator $F_{\omega , \alpha} : L^2_p \to L^2_{p-2}$ is elliptic and self-adjoint, and hence it is Fredholm. Lemma \ref{lemlinojfcsasorlok} also shows that the kernel of $F_{\omega , \alpha}$ must be trivial, and hence by the Fredholm alternative we conclude that $F_{\omega , \alpha} : L^2_p \to L^2_{p-2}$ is surjective.

In other words, for every $f \in L^2_{p-2}$ there exists $\phi \in L^2_p$ such that $F_{\omega , \alpha} (\phi) =f$. However, since $F_{\omega , \alpha}$ is elliptic, $\phi$ must be smooth if $f$ is smooth by the elliptic regularity. This establishes the claim stated in the lemma.

%Thus, Lemma \ref{lemlinojfcsasorlok}, Fredholm alternative, and elliptic regularity prove the result.
\end{proof}

We can repeat the above procedure to get the following result.

\begin{lemma}
Suppose $\Lambda_{\omega} \alpha = \cst$. Then, for each $m \in \mathbb{N}$ there exist $\phi_1 , \dots , \phi_m \in C^{\infty} (X , \rl)$ such that 
\begin{equation*}
\omega_m := \omega + \ai \ddbar (R^{-1}\phi_1 + \cdots + R^{-m} \phi_m)
\end{equation*}
satisfies
\begin{equation*}
S(\omega_m) - R \Lambda_{\omega_m} \alpha = \cst + R^{-m} f_{m,R}
\end{equation*}
for a function $f_{m,R} $ with average $0$ (with respect to $\omega_m$) which is bounded in $C^{\infty}(X,\rl)$ for all sufficiently large $R$.
\end{lemma}

\begin{proof}
We simply expand $S(\omega_m) - R \Lambda_{\omega_m} \alpha$ at $\omega$ to get
\begin{align*}
S(\omega_m) - R \Lambda_{\omega_m} \alpha = \cst &+  (S(\omega) - \bar{S})  + (\alpha , \ai \ddbar \phi_1)_{\omega} \\
&+ \sum_{i=1}^{m-1} \frac{1}{R^i} \left( (\alpha , \ai \ddbar \phi_{i+1})_{\omega} + B_i \right) + O(R^{-m}) ,
\end{align*}
where each $B_i$ is a smooth function with average 0 (with respect to $\omega$) which depends only on $\phi_1 , \dots , \phi_{i}$. Thus, repeated application of Lemma \ref{lemsolvoboR} establishes the claimed result.
\end{proof}

Let $G_{m,R} $ be the solution to $\Delta_{\omega_m} G_{m,R} = f_{m,R}$. By the standard elliptic PDE theory (cf.~\cite{taylorpde}), we see that there exists a constant $C (\omega_m, p)$ depending on $\omega_m$ and $p \in \mathbb{N}$ such that the $L^2_p$-Sobolev norm of $G_m$ can be estimated as
\begin{equation*}
|| G_{m,R} ||_{p, \omega_m} \le C(\omega_m , p) || f_{m,R} ||_{{p-2} , \omega_m} ,
\end{equation*}
where the $L^2_p$-Sobolev norm $|| \cdot ||_{p , \omega_m}$ is defined with respect to $\omega_m$ (note that we may choose $p$ to be a sufficiently large integer). On the other hand, we have $ \omega - \omega_m = O( R^{-1})$, and hence we have
\begin{equation*}
|| G_{m,R} ||_{p, \omega} \le C (\omega , p) || f_{m,R} ||_{{p-2} , \omega} 
\end{equation*}
for a constant $C (\omega , p) >0$ which depends only on $\omega$ and $p$. Since $|| f_{m,R} ||_{{p-2} , \omega}$ can be bounded by a constant uniformly of $R$, we finally see that $|| G_{m,R} ||_{p, \omega}$ can be bounded by a constant uniformly of $R$, for each $p \in \mathbb{N}$.

We now define
\begin{equation*}
\alpha_m := \alpha + \frac{\ai }{R^m} \ddbar G_{m,R} (\omega_{m}),
\end{equation*}
and observe that it satisfies the equation
\begin{equation*}
S(\omega_m) - R \Lambda_{\omega_m} \alpha_m = \cst ,
\end{equation*}
as in Proposition \ref{trivsolprop}. The above argument implies $ \alpha - \alpha_m =O( R^{-m})$, and hence we have obtained the claimed $R \alpha_m$-twisted cscK metric for $\alpha_m$ that is different from $\alpha$ by order $1/ R^m$. The next step in the proof, which we discuss immediately in the next section, is to apply the Banach space inverse function theorem to perturb the $R \alpha_m$-twisted cscK metric $\omega_m$ to an $R \alpha$-twisted cscK metric, similarly to the argument in \cite{finesurfaces} and \cite{lebsim}.

%Since $\omega_m$ moves in a compact subset in $C^{\infty}$ for all sufficiently large $R$, $G_{m,R} (\omega_m)$ should be bounded in $C^{\infty}$. Then observe that
%\begin{equation*}
%S(\omega_m) - R \Lambda_{\omega_m} \alpha_m = \cst 
%\end{equation*}
%if we define
%\begin{equation*}
%\alpha_m := \alpha + \frac{\ai }{R^m} \ddbar G_{m,R} (\omega_{m}),
%\end{equation*}
%as we did in Proposition \ref{trivsolprop}. We have thus constructed the approximate solutions that we seek, and are now ready to apply the inverse function theorem.

\subsection{Applying the inverse function theorem} \label{appinvfnth}

%The argument that follows is similar to \cite{lebsim}. 

We recall the following well-known theorem.

\begin{theorem} 
\emph{(Quantitative inverse function theorem; cf.~\cite[Theorem 4.1]{finesurfaces})} \label{qinvfth}
Suppose that $B_1$, $B_2$ are Banach spaces, $U \subset B_1$ is an open set containing the origin, and that
\begin{enumerate}
\item $T : U \to B_2$ is a differentiable map whose derivative at $0$, $DT |_0$, is an isomorphism of Banach spaces, with the inverse $P$;
\item $\delta'$ is the radius of the closed ball in $B_1$, centred at $0$, on which $T- DT$ is Lipschitz, with constant $1/ (2 ||P||_{op})$;
\item $\delta := \delta' / (2 ||P||_{op})$.
\end{enumerate}
Then whenever $y \in B_2$ satisfies $|| y- T(0) || < \delta$, there exists $x \in U$ with $T(x) = y$. Moreover, such an $x$ is unique subject to the constraint $||x|| < \delta' $.
\end{theorem}

%For notational convenience, we write $\omega'$ for $\omeg_m$ and $\alpha'$ for $\alpha_m$ in what follows.

Suppose that we write, as before, $L^2_p$ for the Sobolev completion of $C^{\infty} (X , \rl)_0$ and $\Omega^{1,1}$ for the set of $(1,1)$-forms on $X$ completed by the $L^2_p$-Sobolev norm $|| \cdot ||_p$. All Sobolev norms\footnote{When we consider the elliptic estimates of the operator $-  \mathfrak{D}_{m}^* \mathfrak{D}_{m} + R F_{m}$ it is natural to use the norm defined by $\omega_m$, but $\omega - \omega_m = O(R^{-1})$ allows us to absorb the difference in the constant in the estimate.} are defined with respect to the \kah metric $\omega$, and we shall take $p>0$ to be sufficiently large. We now take $B_1 :=\Omega^{1,1} \times L^2_{p+4}  $, $U:= \{ (\epsilon, \phi) \in B_1 \mid \omega_m + \ai \ddbar \phi >0 \}$, $B_2 :=\Omega^{1,1} \times  L^2_p$ in Theorem \ref{qinvfth}, and define
\begin{equation} \label{defonftepsphi}
T (\epsilon ,\phi ) : = (\alpha_m + \epsilon ,  S(\omega_{m , \phi} ) - R \Lambda_{\omega_{m , \phi}}( \alpha_m + \epsilon))
\end{equation}
where $\omega_{m , \phi}:= \omega_m + \ai \ddbar \phi$; this means that $0 \in B_1$ is identified with $(\alpha_m , \omega_m) $. Since $T$ depends on $R$, we shall write $T_R$ for $T$ in what follows.

\begin{notation}
For notational convenience, we shall write $\Lambda_m$ for $\Lambda_{\omega_m}$, $\mathfrak{D}_{m}^* \mathfrak{D}_{m}$ for $\mathfrak{D}_{\omega_m}^* \mathfrak{D}_{\omega_m}$, and $F_m$ for $F_{\omega_m , \alpha_m}$ in what follows.
\end{notation}

%and $T (\alpha ,\phi ) : = (\alpha , \frac{1}{R} S(\omega_{\phi}) - \Lambda_{\omega_{\phi}} \alpha )$. 

Since $\omega_m$ is $R \alpha_m$-twisted cscK, the equation (\ref{linatcsckeq}) implies that we have
\begin{equation} \label{lintrepsphi}
DT_R|_{0}  (\epsilon, \phi )=
\begin{pmatrix}
1 & 0 \\
\Lambda_{m} &-  \mathfrak{D}_{m}^* \mathfrak{D}_{m} + R F_{m} 
\end{pmatrix}
\begin{pmatrix}
\epsilon \\
\phi
\end{pmatrix}.
\end{equation}
Lemma \ref{lemlinojfcsasorlok} implies
\begin{equation*}
\int_X \phi ( - \mathfrak{D}_{m}^* \mathfrak{D}_{m} \phi + F_{m} (\phi)) \frac{\omega_m^n}{n!} = - \int_X || \bar{\partial} \mathrm{grad}^{1,0}_{\omega_m} \phi ||^2_{\omega_m}  \frac{\omega_m^n}{n!}   - \int_X || \xi_{\phi}||^2_{\alpha_m} \frac{\omega_m^n}{n!} ,
\end{equation*}
and hence the kernel of $- \mathfrak{D}_{m}^* \mathfrak{D}_{m} + R F_{m} : L^2_{p+4} \to L^2_p$ must be zero. Since this operator is elliptic and $X$ is compact without boundary, the Fredholm alternative implies that $- \mathfrak{D}_{m}^* \mathfrak{D}_{m} + R F_{m} $ is surjective (cf.~Lemma \ref{lemsolvoboR}). Thus $DT_R |_0 : B_1 \to B_2$ is an isomorphism whose inverse $P = P_R$ is given by
\begin{equation*}
P_R  (\epsilon, \phi )=
\begin{pmatrix}
1 & 0 \\
- \left( -  \mathfrak{D}_{m}^* \mathfrak{D}_{m} + R F_{m} \right)^{-1} \Lambda_{m} & \left( -  \mathfrak{D}_{m}^* \mathfrak{D}_{m} + R F_{m} \right)^{-1}
\end{pmatrix}
\begin{pmatrix}
\epsilon \\
\phi
\end{pmatrix}.
\end{equation*}

\begin{remark}
We recall that, in fact, the kernel of the linearisation of $\phi \mapsto S(\omega_{\phi}) -  \Lambda_{\omega_{\phi}} \alpha$ is trivial if $\omega$ is only assumed to be \textit{sufficiently close} to an $\alpha$-twisted cscK metric \cite[Lemma 4.3]{chentwisted}.
\end{remark}

To evaluate the operator norm $||P_R||_{op}$ of the inverse $P_R$, it suffices to evaluate $||(  -  \mathfrak{D}_{m}^* \mathfrak{D}_{m} + R F_{m})^{-1} ||_{op}$. Recalling the proof of the fundamental elliptic estimate (e.g.~\cite[Theorem 11.1 in Chapter 5]{taylorpde}), we see that there exists a constant $C_1= C_1 ( \alpha , \omega , p)$ independent of $R$ such that
\begin{equation*}
|| (  -  \mathfrak{D}_{m}^* \mathfrak{D}_{m} + R F_{m})^{-1} \phi ||_{p+4} \le R C_1 \left( || \phi ||_{p} + ||(  -  \mathfrak{D}_{m}^* \mathfrak{D}_{m} + R F_{m})^{-1} \phi ||_{L^2 } \right),
\end{equation*}
by also recalling $\omega_m -  \omega = O(R^{-1})$ and $\alpha_m - \alpha = O(R^{-m})$.

%To evaluate the operator norm $||P_R||_{op}$ of the inverse $P_R$, it suffices to evaluate the first eigenvalue of $- \frac{1}{R} \mathfrak{D}_{m}^* \mathfrak{D}_{m} + F_{m}$. 

We also have the following lemma.

\begin{lemma}
Let $\lambda_{1,m} <0$ be the largest non-zero eigenvalue of $-  \mathfrak{D}_{m}^* \mathfrak{D}_{m} + RF_{m}$. Then there exists a constant $C_2 = C_2  (\alpha , \omega) >0 $ such that $\lambda_{1,m} < -C_2 R$ for all large enough $R$.
\end{lemma}

\begin{proof}
Let $\phi_{1,m} $ be an eigenfunction corresponding to $\lambda_{1,m}$. Then, by Lemma \ref{lemlinojfcsasorlok}, we have
\begin{align*}
\lambda_{1,m} &= \frac{1}{||\phi_{1,m}||^2_{L^2, \omega_m}} \int_X \phi_{1,m} \left( -  \mathfrak{D}_{m}^* \mathfrak{D}_{m} \phi_{1,m} + R F_{m} (\phi_{1,m}) \right) \frac{\omega_m^n}{n!} \\
&= \frac{1}{||\phi_{1,m}||^2_{L^2 , \omega_m}} \left( -  \int_X || \bar{\partial} \mathrm{grad}^{1,0}_{\omega_m} \phi_{1,m} ||^2_{\omega_m}  \frac{\omega_m^n}{n!}   - R \int_X || \xi_{\phi_{1,m}}||^2_{\alpha_m} \frac{\omega_m^n}{n!} \right) \\
& \le \frac{R}{||\phi_{1,m}||^2_{L^2 , \omega_m}} \left(   - \int_X || \xi_{\phi_{1,m}}||^2_{\alpha_m} \frac{\omega_m^n}{n!} \right) .
%&< C_{\alpha , \omega} \frac{- \int_X || \partial \phi||^2_{\omega} \frac{\omega^n}{n!}}{\int_X ||d \phi||^2_{\omega} \frac{\omega^n}{n!}} \\
%&< - \frac{C_{\alpha , \omega}}{4}
\end{align*}
Now observe that there exists a constant $C_3 = C_3 (\alpha , \omega , m) >0$ such that
\begin{equation*}
C_3  \int_X || \partial \phi_{1,m}||^2_{\omega_m} \frac{\omega_m^n}{n!} < \int_X || \xi_{\phi_{1,m}}||^2_{\alpha_m} \frac{\omega_m^n}{n!},
\end{equation*}
recalling that $\xi_{\phi_{1,m}}$ is the $\omega_m$-metric dual of $\partial \phi_{1,m}$. Since $\omega_m - \omega = O(R^{-1})$ and $\alpha_m - \alpha = O(R^{-m})$, we see that there exists a constant $C_4 = C_4 (  \alpha , \omega) >0$ such that $C_4  \ge C_3 (\alpha , \omega , m)$ uniformly of all large enough $R$. We thus get
\begin{equation*}
\lambda_{1,m} < - \frac{R C_4 }{||\phi_{1,m}||^2_{L^2 , \omega_m}} \int_X || \partial \phi_{1,m}||^2_{\omega_m} \frac{\omega_m^n}{n!} 
\end{equation*}
Now the Poincar\'e inequality yields
\begin{equation*}
||\phi_{1,m}||^2_{L^2 , \omega_m} < C_5 \int_X || d \phi_{1,m}||^2_{\omega_m} \frac{\omega_m^n}{n!} 
\end{equation*}
for a constant $C_5 = C_5 ({\omega, m}) >0$. As before, $\omega_m - \omega = O(R^{-1})$ implies that we can bound $C_5 ({\omega,m})$ by another positive constant $C_6 = C_6 ({\omega})$ uniformly of $R$. Thus we finally get
\begin{equation*}
\lambda_{1,m} < - \frac{ C_4 }{2 C_6} R % - \frac{C_6 ({\omega}) C_4 ({\alpha , \omega})}{2} R
\end{equation*}
as claimed.

%where we used $C_{\alpha, \omega} \int_X || \partial \phi||^2_{\omega} \frac{\omega^n}{n!} > \int_X || \xi_{\phi}||^2_{\alpha} \frac{\omega^n}{n!}$ for some constant $C_{\alpha, \omega}>0$ and the Poincar\'e inequality in the third line.
\end{proof}

Thus, again using $\omega_m - \omega = O(R^{-1})$, we have
\begin{align*}
||P_R  ||_{op} &= \sup_{\epsilon, \phi} \frac{||\epsilon||_{p} +||- \left( -  \mathfrak{D}_{m}^* \mathfrak{D}_{m}  + R F_{m} \right)^{-1} \Lambda_{m} \epsilon +\left( -  \mathfrak{D}_{m}^* \mathfrak{D}_{m} + R F_{m} \right)^{-1} \phi  ||_{p+4}}{||\epsilon||_{p} + ||\phi||_{p } } \\
&\le \sup_{\epsilon, \phi} \frac{||\epsilon||_{p} + || \left( -  \mathfrak{D}_{m}^* \mathfrak{D}_{m}  + R F_{m} \right)^{-1} \Lambda_{m} \epsilon ||_{p+4} + ||\left( -  \mathfrak{D}_{m}^* \mathfrak{D}_{m} + R F_{m} \right)^{-1} \phi  ||_{p+4}}{||\epsilon||_{p} + ||\phi||_{p}} \\
&< 1+ RC' \left( 1+ |\lambda_{1,m} |^{-1}  \right) \sup_{\epsilon , \phi} \frac{ || \Lambda_{m} \epsilon||_{p} + || \phi  ||_{p }}{||\epsilon||_{p} + ||\phi||_{p}} \\
&< 1+2 RC_1  \left(1 + \frac{1}{C_2R}\right) (1 + || \Lambda_{\omega}||_{op}),
\end{align*}
%which is bounded uniformly for all large enough $R>0$.
and hence there exists a constant $C'  = C' (\alpha, \omega, p) >0$ such that $||P_R  ||_{op} \le C' R$.

Recalling the definition (\ref{defonftepsphi}) of $T_R$, we see that for $l \ge 3$, $l \in \mathbb{N}$, and on a ball centred at $0 \in B_1$ with radius $\delta' := R^{-l}$, the operator $T_R - DT_R$ is Lipschitz with constant $1/ (2 ||P_R||_{op})$ for all large enough $R$. Thus we can choose 

%Recalling the definition (\ref{defonftepsphi}) of $T_R$, it is immediate that there exists $\delta' >0$ such that $T_R - DT_R$ is Lipschitz on a ball centred at $0 \in B_1$ with radius $\delta'$, with constant $1/ (2 ||P_R||_{op})$. Thus we can choose 
\begin{equation*}
\delta = \frac{\delta'}{ 2 || P_R||_{op}} > \frac{1}{2 C'} R^{-l-1} ,%> \cst. R \gg 1 ,
\end{equation*}
so that the quantitative inverse function theorem holds in the ball of radius $\delta = O(R^{-l-1})$ in $B_2$ centred at $T_R (0,0)$.

Writing $\bar{S}$ for the average of the scalar curvature and $c$ for the average of $\Lambda_{m} \alpha_m$ (cf.~Remark \ref{remavsclal}), we observe $T_R (0,0) = (\alpha_m ,  \bar{S} - R c)$. Since $\alpha_m - \alpha = O(R^{-m})$, we see that there exists a constant $C'' = C'' (\alpha , \omega, p)$ such that
\begin{equation*}
|| T_R (0 , 0) - (\alpha, \bar{S} -R c) ||_{L^2_p} < C'' R^{-m+1} %< \frac{1}{2C'} R^{-l-1} < \delta
\end{equation*}
for all large enough $R>0$, and note that we have
\begin{equation*}
C'' R^{-m+1} < \frac{1}{2C'} R^{-l-1} < \delta
\end{equation*}
for all large enough $R>0$, by taking $m$ to be sufficiently large. Thus, for all large enough $R>0$, there exists $(\epsilon, \phi) \in U \subset B_1$ such that $T_R (\epsilon , \phi) = (\alpha, \bar{S} - R c)$; in other words we have
\begin{equation*}
\begin{cases}
\alpha_m + \epsilon = \alpha \\
S(\omega_{m , \phi} ) - R \Lambda_{\omega_{m , \phi}} ( \alpha_m + \epsilon) = \bar{S} - Rc = \cst
\end{cases}
\end{equation*}
for some $(\epsilon , \phi ) \in \Omega^{1,1} \times L^2_{p+4}$ (note also that we have $\omega_{m, \phi}>0$ for all large enough $R>0$, since $||\phi||_{p+4} < \delta' = R^{-l}$). By taking $p$ to be sufficiently large and recalling the Sobolev embedding, we can use the elliptic regularity, as in \cite[Lemma 2.3]{finesurfaces}, to conclude that $\phi $ is in $C^{\infty} (X,\rl)$. This establishes all the statements claimed in Theorem \ref{mrtcsck}.

%Since $p$ could be any large enough integer, the uniqueness statement in Theorem \ref{qinvfth} and the natural embedding $L^2 _{p+1} \inj L^2_{p}$ (together with the Sobolev embedding) implies that $\phi $ is in $C^{\infty} (X,\rl)$ (alternatively we can argue as in \cite[Lemma 2.3]{finesurfaces} to conclude the regularity). This establishes all the statements claimed in Theorem \ref{mrtcsck}.

%As $\beta = \alpha$ is obvious, we thus see that there exists $\phi \in L^2_p$ such that $S(\omega_{\phi} ) - \Lambda_{\omega_{\phi}} \alpha = \cst$. 

%\begin{equation*}
%T_R (- R^{-m} \ai \ddbar G_{m,R} , 0) = (\alpha , \frac{1}{R} \bar{S} - c + O(R^{-m}))
%\end{equation*}

%On the other hand, $T_R ( \alpha' ,0) =( \alpha + \frac{1}{R} \ai \ddbar G(\omega),  \bar{S} + c)$, where $\Lambda_{\omega} \alpha=c$. Thus
%\begin{equation*}
%|| T_R (\alpha ', 0) - (\alpha, \bar{S} +c) ||_{L^2_p} < \frac{1}{R} C_{\omega}  < \delta
%\end{equation*}
%once $R>0$ is chosen to be sufficiently large. Thus there exists $(\beta, \phi)$ such that $T_R (\beta , \phi) = (\alpha, \bar{S}+c)$. As $\beta = \alpha$ is obvious, we thus see that there exists $\phi \in L^2_p$ such that $S(\omega_{\phi} ) - \Lambda_{\omega_{\phi}} \alpha = \cst$. Since $p$ could be any positive integer, the uniqueness statement in the inverse function theorem and the natural embedding $L^2 _{p+1} \inj L^2_{p}$ (together with the Sobolev embedding) implies that $\phi $ is in $C^{\infty} (X,\rl)$ (or argue as in \cite[Lemma 2.3]{finesurfaces}).

\section{Proof of Corollary \ref{cormroptwa}}
We now apply the argument in \S \ref{appinvfnth} to prove Corollary \ref{cormroptwa}. Suppose that we have an $\alpha'$-twisted cscK metric $\omega$. Using the notation from \S \ref{appinvfnth}, we define an operator $T' : U \to B_2$ by
\begin{equation*}
T'(\epsilon , \phi) = (\alpha ' + \epsilon , S(\omega_{\phi}) - \Lambda_{\omega_{\phi}}(\alpha' + \epsilon)) .
\end{equation*}
It suffices to show that the linearisation $DT' |_0$ of $T'$ at $0 \in B_1$ is invertible. Since $\omega$ is $\alpha'$-twisted cscK, we can prove the invertibility of $DT' |_0$ by arguing exactly as we did in \S \ref{appinvfnth} to show that the operator (\ref{lintrepsphi}) is invertible. Thus Theorem \ref{qinvfth}, applied for a sufficiently large $p$, immediately implies Corollary \ref{cormroptwa} by recalling the Sobolev embedding and the elliptic regularity (cf.~\cite[Lemma 2.3]{finesurfaces}).
%, which was also proved by X.X.~Chen \cite[Theorem 4.1]{chentwisted}.

%\begin{proof}[Proof of Corollary \ref{cormroptwa}]
%Suppose that we have an $\alpha'$-twisted cscK metric $\omega$. Using the notation from \S \ref{appinvfnth}, we define an operator $T' : U \to B_2$ by
%\begin{equation*}
%T'(\epsilon , \phi) = (\alpha ' + \epsilon , S(\omega_{\phi}) - \Lambda_{\omega_{\phi}}(\alpha' + \epsilon)) .
%\end{equation*}
%It suffices to show that the linearisation $DT' |_0$ of $T'$ at $0 \in B_1$ is invertible. Since $\omega$ is $\alpha'$-twisted cscK, we can prove the invertibility of $DT' |_0$ by arguing exactly as we did in \S \ref{appinvfnth} to show that the operator (\ref{lintrepsphi}) is invertible. Thus Theorem \ref{qinvfth}, applied for a sufficiently large $p$, immediately implies Corollary \ref{cormroptwa} by recalling the Sobolev embedding and the elliptic regularity (cf.~\cite[Lemma 2.3]{finesurfaces}).
%Observe
%\begin{equation*}
 %\int_X \phi ( - \mathfrak{D}_{\omega}^* \mathfrak{D}_{\omega} \phi + F_{\alpha'} (\phi)) \frac{\omega^n}{n!} = - \int_X || \bar{\partial} \mathrm{grad}^{1,0}_{\omega} \phi ||^2_{\omega}  \frac{\omega^n}{n!}   - \int_X || \xi_{\phi}||^2_{\alpha'} \frac{\omega^n}{n!} ,
%\end{equation*}
%which is non-zero if and only if $\phi$ is non-constant.
%\end{proof}

%\section{When $\alpha$ is semipositive}

%Suppose now that $\alpha$ is positive semidefinite. Most of what we discussed in \S \ref{linotcscke} carries over to this case, apart from that the kernel of $F_{\omega , \alpha}$ may be larger than the set of constant functions.

\bibliographystyle{amsplain}
\bibliography{2015_22_twisted}

\begin{flushleft}
DEPARTMENT OF MATHEMATICS, UNIVERSITY COLLEGE LONDON \\
Email: \verb|yoshinori.hashimoto.12@ucl.ac.uk|, \verb|yh292@cantab.net|
\end{flushleft}

\end{document}